\documentclass[12pt,a4paper]{article}
\usepackage{lmodern}
\usepackage{xcolor}
\usepackage{amsthm}
\usepackage{amsmath,amsfonts,amssymb}
\usepackage{anysize}

\marginsize{1.4 in}{1 in}{1 in}{1 in}

\theoremstyle{plain}
\newtheorem{thm}{Theorem}[section]
\newtheorem{cor}[thm]{Corollary}
\newtheorem{lem}[thm]{Lemma}
\newtheorem{prop}[thm]{Proposition}
\theoremstyle{definition}
\newtheorem{defn}{Definition}[section]

\newtheorem{rem}[defn]{Remark}
\numberwithin{equation}{section}

\usepackage{array}
\newcommand{\PreserveBackslash}[1]{\let\temp=\\#1\let\\=\temp}
\newcolumntype{C}[1]{>{\centering\arraybackslash}p{#1}}
\usepackage{tabularx}

\date{}

\begin{document}
    \begin{center}
        {\LARGE \textbf{On the Sum of Additive Characters over Finite Fields}}\\
        {\small Maithri K.$^{1}$, Vadiraja Bhatta G. R.$^{1,*}$, Indira K. P.$^{1}$ \\
        $^{1}$Department of Mathematics, Manipal Institute of Technology, \\
        Manipal Academy of Higher Education, Manipal, India.}
    \end{center}
    \begin{abstract}
        In this paper, we study the sum of additive characters over finite fields, with a focus on those of specified $\mathbb{F}_q$-Order. We establish a general formula for these character sums, providing an additive analogue to classical results previously known for multiplicative characters. As an application, we derive a M\"obius function $\mu(g)$ for polynomials $g \in \mathbb{F}_q[x]$, analogous to the integer M\"obius function $\mu(n)$, and develop a characteristic function for $k$-normal elements. We also generalize several classical identities from the integer setting to the polynomial setting, highlighting the structural parallels between these two domains.
    \end{abstract}
    \section{Introduction}

    Let $\mathbb{F}_{q^m}$ be a finite field of characteristic $p$, where $p$ is a prime. This field can be viewed as an $m$-dimensional vector space over its subfield $\mathbb{F}_q$. An element $\alpha \in \mathbb{F}_{q^m}$ is called \emph{normal} over $\mathbb{F}_q$ if $\alpha$ together with all its conjugates forms a basis of $\mathbb{F}_{q^m}$ over $\mathbb{F}_q$. Such a basis is referred to as a \emph{normal basis}.

    The multiplicative group $\mathbb{F}_{q^m}^*$ is cyclic of order $q^m - 1$, and any element which generates this group is called a \emph{primitive element}. Primitive elements exist in every finite field. Both primitive and normal elements are of significant interest in applications, particularly in cryptography. Elements that are simultaneously primitive and normal, referred to as \emph{primitive normal elements}, are especially valued for their algebraic structure and efficiency in computation.

    The existence of primitive normal elements has been extensively studied. Carlitz first proved their existence for sufficiently large values of $q^m$~\cite{Carlitz1952Prim,CarlitzPrimitiveRootssomeproblem}, and Davenport extended the result to the case where $q$ is prime~\cite{Davenport_existance}. A complete proof covering all cases was later given by Lenstra and Schoof in 1987~\cite{LenstraPriminormal}.

    Expanding on these ideas, Kapetanakis and Reis introduced the notion of \emph{$r$-primitive elements}, which are elements of $\mathbb{F}_{q^m}$ with multiplicative order $\frac{q^m - 1}{r}$~\cite{Reis2019Variation}. In parallel, Huczynska et al.\ introduced \emph{$k$-normal elements}, generalizing the classical notion of normality~\cite{Mullen2013Existance}. These $k$-normal elements when considered together with their conjugates forms the basis of $q$-modulus of dimension $m-k$. Later studies looked at whether $r$-primitive and $k$-normal elements exist and what properties they have, for different choices of $r$ and $k$~\cite{Lucas2,Neumann}.

    Several studies have further explored related aspects, such as the behavior of inverses of these elements~\cite{RANI2_inverse}, elements with prescribed norm and trace~\cite{RANI1_norm&trace}, and the existence of such elements occurring in pairs~\cite{Aguirre2_pairs_rk} or forming arithmetic progressions~\cite{Aguirre3_AP_rk,Laishram_AP}. These works underscore the depth of the topic and reveal its close ties to the theory of additive and multiplicative characters, which serve as essential tools in analyzing the arithmetic structure of finite fields.

    The study of $r$-primitive elements typically involves multiplicative characters, whereas $k$-normal elements are often approached using additive characters. A key quantity that arises in the construction of characteristic functions particularly for $g$-free elements and those with a prescribed $\mathbb{F}_q$-Order is the sum $\displaystyle \sum_g \chi(\alpha)$, where $\chi$ ranges over additive characters of $\mathbb{F}_q$-Order $g$. This sum plays a central role in numerous works, including~\cite{Mullen2013Existance, Reis2019Variation, RANI2_inverse, Aguirre2_pairs_rk, RANI1_norm&trace}, and serves as the basis for the present study. The foundational analysis of such sums can be found in~\cite{Carlitz1954Sum}.

    We focus on the analysis of this sum of additive characters. As preparation, we adopt the following notation for polynomials over $\mathbb{F}_q[x]$: the Euler phi function $\phi(f)$ counts the number of polynomials of degree less than $\deg(f)$ and coprime to $f$; the Möbius function $\mu(f)$ is defined as $(-1)^r$ if $f$ is the product of $r$ distinct irreducible polynomials and zero otherwise; $W(f)$ denotes the number of square-free divisors of $f$, which equals $2^{w(f)}$, where $w(f)$ is the number of irreducible factors of $f$.

    The structure of the paper is as follows. In Section~2, we present the necessary background and preliminaries. Section~3 establishes a general result on the sum of additive characters with a given $\mathbb{F}_q$-Order. As a consequence, we obtain a formula for the polynomial Möbius function $\mu(g)$, which parallels the classical Möbius function $\mu(n)$ for integers. In Section~4, we apply this formula to construct a characteristic function for elements with a specific $\mathbb{F}_q$-Order, and thereby for $k$-normal elements. 

    \section{Preliminary Definitions}
    
    Consider a field $\mathbb{F}_{q^m}$ over $\mathbb{F}_q$. R.~Lidl and H.~Niederreiter~\cite{LIDL} provided an equivalent condition for noraml element: an element $\alpha \in \mathbb{F}_{q^m}$ is normal if and only if the polynomials $x^m - 1$ and $\alpha x^{m-1} + \alpha^q x^{m-2} + \cdots + \alpha^{q^{m-2}} x + \alpha^{q^{m-1}}$ are coprime.

    Motivated by this characterization, S.~Huczynska et al.~\cite{Mullen2013Existance} introduced a broader notion called $k$-normal elements.
    \begin{defn}\cite{Mullen2013Existance} 
    
        Let $\alpha \in \mathbb{F}_{q^m}$ and $g_\alpha(x)$ be a polynomial $\sum\limits_{i=0}^{m-1} \alpha^{q^i} x^{m-1-i} \in \mathbb{F}_{q^m}[x]$. If gcd$\left(x^m-1, g_\alpha\left(x\right) \right)$ over $\mathbb{F}_{q^m}$ has degree $k$ (where $0 \leq k \leq m-1$ ), then $\alpha$ is a $k$-normal element of $\mathbb{F}_{q^m}$ over $\mathbb{F}_q$.
    \end{defn}

    For any polynomial $l(x) = \displaystyle \sum_{i=0}^n a_i x^i$ over the field $\mathbb{F}_q$, the notation $l \circ x$ is used to denote the associated \emph{linearized polynomial} $L(x) = \displaystyle \sum_{i=0}^n a_i x^{q^i}$. Based on this concept, the $\mathbb{F}_q$-Order of an element is defined as the unique monic polynomial $f$ of least degree such that $f \circ \alpha = 0$.

    Analogous to the $\mathbb{F}_q$-Order of elements, the $\mathbb{F}_q$-Order of additive characters was introduced by Reis~\cite{Reis2020additive}. 
    \begin{defn}\cite{Reis2020additive}
        The $\mathbb{F}_q$-Order of and additive character $\chi \in \displaystyle \widehat{\mathbb{F}_{q^m}}$ is the unique monic polynomial in $f \in \mathbb{F}_q[x]$ of least degree such that $f \circ \chi = \chi_0$, that is $\chi(f\circ \gamma ) = 1, ~\forall \gamma \in \mathbb{F}_{q^m}$. 
    \end{defn}

    Carlitz made several foundational contributions to this area in the 1950s. Some of his key results that are relevant to the present study are recalled below. In his work, Carlitz used the notation $c_a(\xi)$ to denote the sum $\sum_{a(x)} \chi(\xi)$, where $a$ is a polynomial over $\mathbb{F}_q$ dividing $x^m - 1$, and the sum is taken over all additive characters of $\mathbb{F}_q$-Order $a$.

    \begin{lem}\cite{Carlitz1954Sum}\label{CarlitzLemrelativelyprimepol}
        Let $a_1\left|x^m-1 \text{ and } a_2\right| x^m-1,\left(a_1, a_2\right)=1$. Then, $$c_{a_1 a_2}(\xi)=c_{a_1}(\xi) c_{a_2}(\xi).$$ 
        \end{lem}
    \begin{lem} \cite{Carlitz1954Sum}\label{CarlitzLemmairpol}
        Let $u$ be an irreducible polynomial $\in \mathbb{F}_q[x], ~u^e \mid x^m-1$. Let $\xi$ belong to $K(x)$, where $K(x)$ corresponds to $k(x)$; also put $x^m-1=k l$. Then we have $$c_{u^e}(\xi)= \begin{cases}|u|^e-|u|^{e-1} & \left(u^e \mid l\right) \\ -|u|^{e-1} & \left(u^{e-1} \mid l, u^e\nmid l\right) \\ 0 & \left(u^{e-1}\nmid l\right)\end{cases}$$
        where $|u| = q^{deg(u(x))}$.
    \end{lem}
    \section{Sum of additive characters}
    In this section, we study sums of additive characters over finite fields. As mentioned earlier, we first examine these sums in the context of integers, and subsequently extend the analysis to the polynomial setting.

    It is well known that the sum of the $k$th powers of the primitive $n$th roots of unity is given by $\mu(g) \frac{\phi(n)}{\phi(g)}$, where $g = \frac{n}{\gcd(k, n)}$. This result is useful for working with multiplicative character sums.

    \begin{thm}
        If $\alpha$ is primitive element in the field $\mathbb{F}_{q^m}$, then for a divisor $d$ of $q^m-1$, $$\displaystyle \sum_{(d)} \psi_d(\alpha^r) = \mu \left( \dfrac{d}{gcd(d,r)} \right) \dfrac{\phi(d)}{\phi\left( \frac{d}{gcd(d,r)} \right)},$$ where the sum is taken over all multiplicative characters $\psi$ of $\mathbb{F}_{q^m}$ of order $d$.
    \end{thm}
    \begin{proof}
        Let $\alpha$ be a primitive element of $\mathbb{F}_{q^m}$, so that every element of the multiplicative group $\mathbb{F}_{q^m}^*$ can be expressed as a power of $\alpha$. Each multiplicative character $\psi$ of order $d$ satisfies $\psi(\alpha) = \zeta$, where $\zeta$ is a primitive $d$-th root of unity.\\
        Then, for each such character, $\psi(\alpha^r) = \zeta^r$. Hence, the sum  $\displaystyle \sum_{(d)} \psi_d(\alpha^r)$, corresponds to the sum of the $r$-th powers of all primitive $d$-th roots of unity.
        By the formula for sum of $k$-th powers of the primitive $n$-th roots of unity, it follows that, 
         \[\sum_{(d)} \psi_d(\alpha^r) = \mu\left( \frac{d}{\gcd(d, r)} \right) \frac{\phi(d)}{\phi\left( \frac{d}{\gcd(d, r)} \right)}.\]
    \end{proof}

    An analogous result can be formulated in the setting of additive characters. In the multiplicative case, a primitive element is used, as every nonzero element of the field can be expressed as a power of such an element. In contrast, the additive analogue relies on the notion of a normal element, since any element $\beta \in \mathbb{F}_{q^m}$ can be written as $f \circ \alpha$ for some normal element $\alpha$.

    Lemma~\ref{CarlitzLemmairpol} provides a formula for the sum of additive characters evaluated at a fixed element of the field, where the characters range over all those of $\mathbb{F}_q$-Order equal to a power of an irreducible polynomial. The following result gives a more general form of this identity.

    \begin{thm}\label{Sumaddi}
        Let $x^m-1 = f_1(x)f_2(x), ~g(x) \vert x^m-1$ and let $\alpha \in \mathbb{F}_{q^m}$ with $\mathbb{F}_q$-Order $f_1$, then
        $\displaystyle\sum_{g} \chi (\alpha) = \mu \left (d \right)\frac{\phi(g)}{\phi\left(d\right)}$, where $d(x) = \dfrac{g(x)}{gcd \left(g(x),f_2\right)}$ and the summation runs over all the additive characters of $\mathbb{F}_q$-Order $g$. 
    \end{thm}
    \begin{proof}
        Let \(g = g_1^{e_1}(x)g_2^{e_2}(x)\dots g_k^{e_k}(x)\). Then
        \[
        d = \frac{g}{\gcd(g, f_2)} = \frac{g_1^{e_1}g_2^{e_2}\dots g_k^{e_k}}{\gcd(g_1^{e_1}g_2^{e_2}\dots g_k^{e_k}, f_2)} = \frac{g_1^{e_1}g_2^{e_2}\dots g_k^{e_k}}{g_1^{l_1}g_2^{l_2}\dots g_k^{l_k}} = g_1^{e_1-l_1}g_2^{e_2-l_2}\dots g_k^{e_k-l_k}.\]
        Moreover, by Lemma~\ref{CarlitzLemrelativelyprimepol}, we have
        \begin{equation}
            \sum_{g} \chi(\alpha) = \sum_{g_1^{e_1}} \chi(\alpha) \sum_{g_2^{e_2}} \chi(\alpha) \dots \sum_{g_k^{e_k}} \chi(\alpha).
            \label{multiplicity_of_sum}
        \end{equation}

        \textbf{Case 1:} \(d = 1\)

        In this case, \(g\) divides \(f_2\), that is, \(g_i^{e_i} \mid f_2\) for all \(1 \leq i \leq k\). By Lemma~\ref{CarlitzLemmairpol}, it follows that
        \[
        \sum_{g_i^{e_i}} \chi(\alpha) = |g_i|^{e_i} - |g_i|^{e_i-1} = \phi(g_i^{e_i}), \quad \text{for all } 1 \leq i \leq k.
        \]
        Substituting into equation~\eqref{multiplicity_of_sum}, we obtain
        \[
        \sum_{g} \chi(\alpha) = \prod_{i=1}^{k} \phi(g_i^{e_i}) = \phi(g) = \mu(d)\frac{\phi(g)}{\phi(d)}.
        \]

        \textbf{Case 2:} \(d = g_1g_2\dots g_r\)

        By Lemma~\ref{CarlitzLemmairpol}, we have
        \[
        \sum_{g_i^{e_i}} \chi(\alpha) = -|g_i|^{e_i-1} = \mu(g_i) \frac{\phi(g_i^{e_i})}{\phi(g_i)} \quad \text{for } 1 \leq i \leq r,
        \]
        and
        \[
        \sum_{g_i^{e_i}} \chi(\alpha) = \phi(g_i^{e_i}) \quad \text{for } \, r+1 \leq i \leq k. 
        \]
        Substituting these into equation~\eqref{multiplicity_of_sum}, it follows that
        \[
        \sum_{g} \chi(\alpha) = \prod_{i=1}^{r} \mu(g_i) \frac{\phi(g_i^{e_i})}{\phi(g_i)} \prod_{i=r+1}^{k} \phi(g_i^{e_i}) = \mu(d) \frac{\phi(g)}{\phi(d)}.
        \]

        \textbf{Case 3:} Suppose there exist an square factor in d.

        Let $g_i^2 \vert d$ for some $i$, $1\leq i \leq k$, then by lemma~\ref{CarlitzLemmairpol}, we have
        \[
        \sum_{g_i^{e_i}} \chi(\alpha) = 0.
        \]
        Substituting this into equation~\eqref{multiplicity_of_sum}, it follows that
        \[\sum_{g} \chi(\alpha) = 0 = \mu(d) = \mu(d) \frac{\phi(g)}{\phi(d)}.\]
    \end{proof}
    \begin{cor}
        Let $\alpha$ be a normal element in the field $\mathbb{F}_{q^m}$, $f(x) \in \mathbb{F}_q[x]$ and $g(x) \mid x^m-1$, then $\displaystyle \sum_g \chi(f\circ \alpha) = \mu(d) \frac{\phi(g)}{\phi(d)}$, where $d=\dfrac{g}{\text{gcd}(g,f)}$.
    \end{cor}
    \begin{rem}
        It is well-known that for any positive integer $n$, sum of primitive $n$-th roots of unity is given by $\mu(n)$ i.e. $\displaystyle\sum_n \psi_n(\alpha) = \mu\left( n \right) $ where $\alpha$ is a primitive element.  From the above result, we get that for any polynomial $g\vert x^m-1$ over  $\mathbb{F}_q[x],~  \displaystyle \sum_g \chi(\alpha)=\mu(g)$, where $\alpha$ is a normal element.
    \end{rem}

    \section{Role of Sum of Additive Characters in Building Characteristic Function for $k$-normal elements}
    In this section, we establish some basic results on summations involving the Euler's totient function $\phi$ applied to polynomials, which will be utilized in the characteristic function for $k$-normal elements.

    \begin{lem}\label{Lemmasumphi}
         If $u$ is an irreducible polynomial over $\mathbb{F}_q[x]$ of degree $n$, then $\displaystyle \sum_{i=0}^l\phi(u^i) = \frac{\phi(u^{l+1})}{\phi(u)}$ for any positive integer $l$.
     \end{lem}
     \begin{proof}
         \begin{align*}
             \displaystyle \sum_{i=0}^l\phi(u^i) &= 1+\phi(u)+\phi(u^2)+\dots +\phi(u^{l})\\
             &=1+(q^n-1)+(q^{2n}-q^n)+\dots+(q^{ln}-q^{(l-1)n}) = q^{ln} = \frac{q^{ln}(q^n-1)}{q^n-1} \\
             &= \frac{q^{n(l+1)}-q^{ln}}{q^n-1} = \frac{\phi(u^{l+1})}{\phi(u)}.
         \end{align*}
     \end{proof}

    \begin{lem}\label{Lemmaphi(hg)}
        Let $f(x) \in \mathbb{F}_q[x]$ be a divisor of $x^m-1$. For any square-free divisor $h(x)$ of $f, \displaystyle \sum_{\substack{g(x) \vert \frac{x^m-1}{f}\\ gcd(h,\frac{x^m-1}{fg}) =1} } \phi(hg) = q^{deg(\frac{x^m-1}{f})}\phi(h)$.
    \end{lem}
    \begin{proof}
        The factorization of $x^m - 1$ over $\mathbb{F}_q$ can be expressed as a product of the powers of distinct irreducible polynomials, where each irreducible factor appears with the same exponent. Following the approach in~\cite{alizadeh2017some}, we write
        $x^m - 1 = \displaystyle \prod_{i=1}^r p_i^a(x),$
        where each $p_i(x)$ is an irreducible polynomial over $\mathbb{F}_q$, and $a$ is a positive integer determined by the multiplicity structure of the factorization.  
        Let $f = \displaystyle \prod_{i=1}^rp_i^{b_i}(x) \text{ with  } 0\leq b_i\leq a, h = \prod_{i=1}^rp_i^{c_i}(x) \text{ with  } c_i = 0 \text{ or } 1 \text{ and } g \text{ is of the form } g=\prod_{i=1}^rp_i^{d_i}(x) \text{ with  } 0 \leq d_i \leq a-b_i $ and degree of $p_i$ be $m_i$.\\
         Since gcd$(h,\dfrac{x^m-1}{fg}) = 1$, if $c_i=1$ then $d_i=a-b_i$. Let $S = \{i: c_i = 1\}, T = \{j: c_j = 0\}$. \\
         Consider, $\phi(hg) = \displaystyle\prod_{i=1}^r \phi(p_i^{c_i+d_i}) = \prod_{i \in S} \phi(p_i^{1+a-b_i}) \prod_{j \in T} \phi(p_j^{d_j})$, then
         $$\displaystyle \sum_{\substack{g \vert \frac{x^m-1}{f}\\ 
             gcd(h,\frac{x^m-1}{fg}) =1} } \phi(hg) =  \prod_{i \in S} \phi(p_i^{1+a-b_i}) \sum_{\substack{j\in T \\ d_j = 0}}^{a-b_j}\prod_j \phi(p_j^{d_j})$$
             By interchanging summation and product and simplifying $\displaystyle \sum_{d_j=0}^{a-b_j}\phi(p_j^{d_j}) $ we get,
             $$\displaystyle \sum_{\substack{g \vert \frac{x^m-1}{f}\\ 
             gcd(h,\frac{x^m-1}{fg}) =1} } \phi(hg) = \prod_{i \in S} \phi(p_i^{1+a-b_i}) \prod_{j \in T}q^{m_j(a-b_j)} = \prod_{i \in S} q^{m_i(a-b_i)}(q^{m_i}-1) \prod_{j \in T}q^{m_j(a-b_j)}  $$
             Further by using the formula for Euler-phi function of a polynomial we get,
             $$\displaystyle \sum_{\substack{g \vert \frac{x^m-1}{f}\\ 
             gcd(h,\frac{x^m-1}{fg}) =1} } \phi(hg) = \prod_{j=1}^r q^{m_j(a-b_j)}\prod_{i \in S} (q^{m_i}-1) = q^{deg(\frac{x^m-1}{f})}\phi(h).$$
    \end{proof}
    \begin{rem}
        Note that Lemmas 2.2 and 2.3 from \cite{OrderRelated}, originally stated for positive integers, can be extended to the polynomial setting by using Lemma~\ref{Lemmaphi(hg)}.
    \end{rem}
    
     S. D. Cohen \cite{OrderRelated}  provided a characteristic function for elements with a specific multiplicative order. Motivated by this, we derive a characteristic function for elements with a specific \(\mathbb{F}_q\)-Order.
     
    Let $\eta_f$ be the characteristic function for elements with $\mathbb{F}_q$-Order $f$. For any $f \in \mathbb{F}_q[x]$ dividing $x^m-1$ and $\alpha \in \mathbb{F}_{q^m}$, 
     $$\eta_f(\alpha) = \begin{cases}
         1 & \text{ if } \mathbb{F}_q\text{-Order of } \alpha \text{ is } f\\
             0 & \text{ otherwise.}
     \end{cases}$$
     \begin{thm}\label{MainThm}
         Let $f(x)$ be a divisor of $x^m-1$ of degree $n$ over $\mathbb{F}_q[x]$. Then, for any $\alpha \in \mathbb{F}_{q^m}$,
         $$\eta_f(\alpha) = \displaystyle \frac{\phi(f)}{q^m}\sum_{h\vert f} \frac{\mu(h)}{\phi(h)} \sum_{\substack{g \vert \frac{x^m-1}{f}\\ gcd(h,\frac{x^m-1}{fg}) =1} } \sum_{hg} \chi(\alpha) $$  where the inner sum is running over all the additive characters of $\mathbb{F}_q$-Order $hg$.
     \end{thm}
     \begin{proof}
        Let $k(x)$ be the $\mathbb{F}_q$-Order of $\alpha$ and $x^m-1=(p_1(x)p_2(x)\dots p_r(x))^a, f=p_1^{b_1} p_2^{b_2}\dots p_r^{b_r}, k=p_1^{c_1} p_2^{c_2}\dots p_r^{c_r}$. 

        By using theorem \ref{Sumaddi}, 
        \begin{equation}\label{EQThm}
            \eta_f(\alpha) = \displaystyle \frac{\phi(f)}{q^m}\sum_{h\vert f} \frac{\mu(h)}{\phi(h)} \sum_{\substack{g \vert \frac{x^m-1}{f}\\ gcd(h,\frac{x^m-1}{fg}) =1} } \mu \left( \frac{gh}{\text{gcd}(gh,\frac{x^m-1}{k})}\right)\frac{\phi(gh)}{\phi\left( \frac{gh}{\text{gcd}(gh,\frac{x^m-1}{k})}\right)}
        \end{equation}
         
        The following three mutually exclusive cases cover all possibilities:\\
         \textbf{Case 1:} $k(x)=f(x)$\\
         
         Since gcd$(h,\frac{x^m-1}{fg}) = 1$, gcd$(hg,\frac{x^m-1}{f})=g$. The equation \ref{EQThm} reduces to,
         $$\eta_f(\alpha)= \frac{\phi(f)}{q^m}\sum_{h\vert f} \frac{\mu(h)}{\phi(h)} \sum_{\substack{g \vert \frac{x^m-1}{f}\\ gcd(h,\frac{x^m-1}{fg}) =1} } \mu(h) \frac{\phi(gh)}{\phi(h)} $$
         By using lemma \ref{Lemmaphi(hg)}, the above expression becomes
         \begin{align*}
             \eta_f(\alpha) &=\frac{\phi(f)}{q^m}\sum_{h\vert f} \frac{(\mu(h))^2}{(\phi(h))^2}q^{m-n}\phi(h)=\frac{\phi(f)}{q^m}q^{m-n}\sum_{h\vert f} \frac{(\mu(h))^2}{\phi(h)} \\
             & =\frac{\phi(f)}{q^m}q^{m-n}\frac{q^n}{\phi(f)} = 1.
         \end{align*}
         \textbf{Case 2:} $k \mid f$ and $k \neq f$\\
         Without loss of generality we assume that $f \neq 1$, we take $b_r \neq 0$.\\
        Let $S_h = \displaystyle\frac{\mu(h)}{\phi(h)}\sum_{\substack{g \vert \frac{x^m-1}{f} \\ \text{gcd}(h,\frac{x^m-1}{fg})=1}} \mu\left( \frac{gh}{\text{gcd}(gh,\frac{x^m-1}{k})}\right)\frac{\phi(gh)}{\phi\left( \frac{gh}{\text{gcd}(gh,\frac{x^m-1}{k}}\right)}$.\\
        Suppose, $c_i<b_i ~ \forall i, 0 \leq i \leq r.$
        Since $g \vert \frac{x^m-1}{f}, c_i<b_i$ and $h$ is a square-free divisor of $f$, it follows that $gh \vert \frac{x^m-1}{k}$. Thus, gcd$(gh,\frac{x^m-1}{k})= gh.$\\
        Then $S_h = \dfrac{\mu(h)}{\phi(h)} \displaystyle \sum_{\substack{g \vert \frac{x^m-1}{f} \\ \text{gcd}(h,\frac{x^m-1}{fg})=1}} \phi(gh) = \mu(h) q^{m-n}$.\\
        Now for any divisor $h$ of $f$, let us define $h^c = \begin{cases}
            \frac{h}{p_r} \text{ if } p_r \vert h \\
            hp_r \text{ otherwise}.
        \end{cases}$\\
        Note that both $h$ and $h^c$ are square-free divisors of $f$, and they are distinct. Moreover, we have $\mu(h^c) = -\mu(h),$
        thus $S_h + S_{h^c} = q^{m-n}(mu(h) + mu(h^c)) = 0$.\\
        Thus, all such terms cancel in pairs, and the total contribution to $\eta_f(\alpha)$ is zero.\\
        Suppose, $c_1=b_1, c_2 = b_2, \dots, c_j=b_j$ for some $j,~1\leq j \leq r$. Note that at least one $c_j \neq b_j$. Let us take $c_r \neq b_r$. \\
        Now, gcd$\left( gh , \dfrac{x^m-1}{k}\right) = \dfrac{gh}{\text{gcd}(h,p_1p_2\dots p_j)}$. Then,
        \begin{align*}
            S_h &= \frac{\mu(h)}{\phi(h)} \sum_{\substack{g\vert \frac{x^m-1}{f}\\\text{gcd}(h,\frac{x^m-1}{fg})=1}} \frac{\mu(\text{gcd}(h,p_1p_2\dots p_j))}{\phi(\text{gcd}(h,p_1p_2\dots p_j))} \phi(hg) \\
            &= \frac{\mu(h)\mu(\text{gcd}(h,p_1p_2\dots p_j))}{\phi(\text{gcd}(h,p_1p_2\dots p_j))}q^{m-n}.
        \end{align*}
        Again, we pair each $S_h$ with a unique $S_{h^c}$, using the same rule for $h^c$ as above (adding/removing the fixed prime $p_r$). This time, note that
        \[
        \gcd(h^c, p_1p_2\dots p_k) = \gcd(h, p_1p_2\dots p_k),
        \]
        so the pre-factor is the same in both $S_h$ and $S_{h^c}$, and again\\
        $S_h + S_{h^c} = q^{m-n} \dfrac{\mu(\text{gcd}(h,p_1p_2\dots p_k))}{\phi(\text{gcd}(h,p_1p_2\dots p_k))} (\mu(h)+\mu(h^c)) = 0$.\\
        Thus, $\eta_f(\alpha) = 0$.\\
        \textbf{Case 3:} $k \nmid f$\\
        Then there exists an irreducible polynomial $u(x)$ such that $l$ is the highest power of $u$ dividing $k$ and $u^l$ does not divide $f$.\\
        Let $a$ be the highest power of $u$ dividing $x^m-1$, $r$ be the highest power of $u$ dividing $f$ and $h, g$ be fixed with $u \nmid g$, $h$  square-free. Let,
        $$S_g = \mu\left(\dfrac{gh}{\text{gcd}(gh,\frac{x^m-1}{k})}\right)\dfrac{\phi(gh)}{\phi\left(\dfrac{gh}{\text{gcd}(gh,\frac{x^m-1}{k})}\right)}.$$
        Consider the sum, $S = \displaystyle \sum_{i=0}^{a-r}S_{u^ig} = \sum_{i=0}^{a-r} \mu\left(\frac{u^igh}{\text{gcd}(u^igh,\frac{x^m-1}{k})}\right)\frac{\phi(u^igh)}{\phi\left(\frac{u^igh}{\text{gcd}(u^igh,\frac{x^m-1}{k})}\right)}$. \\
        Suppose $u\nmid h$ then,  $\text{gcd}(u^igh,\frac{x^m-1}{k}) = u^i \text{gcd}(gh,\frac{x^m-1}{k})$ for $0 \leq i \leq a-l$ and $\text{gcd}(u^igh,\frac{x^m-1}{k}) = u^{a-l} \text{gcd}(gh,\frac{x^m-1}{k})$ for $a-l < i \leq a-r$. Thus,
        
        \begin{align*}
            S =& \displaystyle \sum_{i=0}^{a-l} \mu \left( \frac{gh}{\text{gcd}(gh,\frac{x^m-1}{k})}\right) \frac{\phi(u^i)\phi(gh)}{\phi\left( \frac{gh}{\text{gcd}(gh,\frac{x^m-1}{k})}\right)} \\&+ \sum_{j=a-l+1}^{l} \mu(u^{j-a+l})  \mu \left( \frac{gh}{\text{gcd}(gh,\frac{x^m-1}{k})}\right) \frac{\phi(u^j)\phi(gh)}{\phi(u^{j-a+l})\phi\left( \frac{gh}{\text{gcd}(gh,\frac{x^m-1}{k})}\right)}   \\
            &= \mu \left( \frac{gh}{\text{gcd}(gh,\frac{x^m-1}{k})}\right) \frac{\phi(gh)}{\phi\left( \frac{gh}{\text{gcd}(gh,\frac{x^m-1}{k})}\right)}\left[1+\phi(u)+\phi(u^2)+\dots+ \phi(u^{a-l})-\frac{\phi(u^{a-l+1})}{\phi(u)}\right]
        \end{align*}
        By lemma \ref{Lemmasumphi}, we get $S=0$.\\
        
        Suppose $u \mid h$, as gcd$(h,\frac{x^m-1}{fg}) = 1, u^{a-r}$ must divide $g$, a contradiction to the choice of $g$ such that $u \nmid g$.\\           
        Thus, $\eta_f = 0$ .\\
        Similarly, if $u \mid g$, write $g = u^j t$ with $\gcd(t,u)=1$.  Replacing $g$ by $t$ in the above argument shows again that all powers of $u$ would have to divide $t$, contradicting $u\nmid t$.  Therefore this case also yields $\eta_f = 0$.

     \end{proof}
     \begin{thm}
         Consider the field $\mathbb{F}_{q^m}$, for any $\alpha \in \mathbb{F}_{q^m}$, and non-negative integer $k\leq m$,
         $$\zeta_k(\alpha) := \displaystyle    \sum_{\substack{f\vert x^m-1 \\ deg(f)=m-k}}\frac{\phi(f)}{q^m}\
         \sum_{h\vert f} \frac{\mu(h)}{\phi(h)} \sum_{\substack{g \vert \frac{x^m-1}{f}\\ gcd(h,\frac{x^m-1}{fg}) =1} } \sum_{hg} \chi(\alpha) =
         \begin{cases}
             1 & \text{ if } \alpha \text{ is } k- \text{normal } \\
             0 & \text{ otherwise}
        \end{cases}$$ is the characteristic function for $k$-normal elements.
    \end{thm}
    We use the fact that an element $\alpha \in \mathbb{F}_{q^m}$ is $k$-normal if and only if degree of $\mathbb{F}_q$-Order of $\alpha$ is of degree $m-k$, and the theorem \ref{MainThm} for proving the above result.
      
    We now present a result that relate the \(\mathbb{F}_q\)-Order, additive characters, and the Euler phi function, in analogy with the integer setting discussed in \cite{Laishram_AP}.
    \begin{prop}
        Consider the field $\mathbb{F}_{q^m}$ over $\mathbb{F}_q$, for any divisor $g$ of $x^m-1$, $$\displaystyle \sum_{\substack{f \in \mathbb{F}_q[x] \\ deg(f) < m}} \mid \sum_g \chi(f \circ \alpha) \mid = q^{m-deg(g)}\phi(g) W(g),$$
        where $\alpha$ is normal element over the field and inner sum is running over all the additive characters of $\mathbb{F}_q$-Order $g$.
    \end{prop}
    \begin{proof}
        \begin{align*}
             \displaystyle \sum_{\substack{f \in \mathbb{F}_q[x] \\ deg(f) < m}} \mid \sum_g \chi(f \circ \alpha) \mid 
             &= \sum_{\substack{f \in \mathbb{F}_q[x] \\ deg(f) < m}} \mid \mu(\frac{g}{\text{gcd}(g,f)})\mid \frac{\phi(g)}{\phi(\frac{g}{\text{gcd}(g,f)})}\\
             &= \phi(g) \sum_{\substack{f \in \mathbb{F}_q[x] \\ deg(f) < m}} \frac{\left (\mu(\frac{g}{\text{gcd}(g,f)})\right)^2}{\phi(\frac{g}{\text{gcd}(g,f)})}\\
             &= \phi(g)\sum_{d \mid g} \frac{\left( \mu(\frac{g}{d}) \right)^2}{\phi(\frac{g}{d})} \sum_{\substack{h \in \mathbb{F}_q[x]\\ deg(h)<m-deg(d) \\ \text{gcd}(h,\frac{g}{d})=1 }} 1 \\
         \end{align*}
         By applying division algorithm for $h$ and $\frac{g}{d}$, $h$ is of the form $\frac{g}{d}s + t$ with $\text{gcd}(k,t) = 1$ and $deg(s) < m-deg(g)$, then \\
         $\displaystyle \sum_{\substack{h \in \mathbb{F}_q[x]\\ deg(h)<m-deg(d) \\ \text{gcd}(h,\frac{g}{d})=1 }} 1 = \sum_{\substack{\frac{g}{d}s+t \\ (\frac{g}{d},t) = 1 \\ deg(d) < m- deg(g)}} 1 = \phi\left(\frac{g}{d}\right)q^{m-deg(g)}$\\
         Thus, $\displaystyle \sum_{\substack{f \in \mathbb{F}_q[x] \\ deg(f) < m}} \mid \sum_g \chi(f \circ \alpha) \mid =  \phi(g) \sum_{d \mid g} \frac{\left( \mu(\frac{g}{d}) \right)^2}{\phi(\frac{g}{d})} \phi\left(\frac{g}{d}\right)q^{m-deg(g)}\\ = \phi(g)q^{m-deg(g)} \sum_{d \mid g} \left( \mu(\frac{g}{d}) \right)^2 = \phi(g) q^{m-deg(g)} W(g)$.
    \end{proof}
    \section{Conclusion}
    In this work, we explored the sum of additive characters over finite fields by organizing them according to their \(\mathbb{F}_q\)-Order. We derived a general formula for such sums, which serves as an additive analogue of known results involving multiplicative characters. As direct applications, we obtained a formula for the Möbius function of polynomials and constructed a characteristic function for \(k\)-normal elements. Additionally, several identities traditionally formulated for integers were successfully extended to the polynomial case, demonstrating the depth of analogy between number-theoretic and polynomial-theoretic structures.

\end{document}